\documentclass{amsart}
\usepackage[numbers]{natbib}

\usepackage{amsmath,amstext, amsthm, amssymb,color,enumitem,amscd}

\usepackage[english]{babel}
\usepackage[all]{xy}

\definecolor{e-mail}{rgb}{0,.40,.80}
\definecolor{reference}{rgb}{.20,.60,.22}
\definecolor{citation}{rgb}{0,.40,.80}

\usepackage[colorlinks, linkcolor=reference,
            citecolor=citation,
           urlcolor=e-mail]{hyperref}

\newtheorem{theorem}{Theorem}[section]

\theoremstyle{remark}
\theoremstyle{definition}

\newtheorem{definition}[theorem]{Definition}

\newtheorem{example}[theorem]{Example}

\numberwithin{equation}{section}

\def\beq{\begin{equation}}
\def\eeq{\end{equation}}

\definecolor{todo}{rgb}{1,0,0}

\definecolor{bl}{rgb}{0,0,1}

\def\Q{{\mathbb Q}}

\def\I{{\mathbb I}}

\def\U{\mathcal{U}}

\def\wt{\widetilde}

\def\cC{{\mathcal C}}

\def\cS{{\mathcal S}}

\def\ol{\overline}

\def\Hom{\operatorname{Hom}}

\def\GL{\operatorname{GL}}

\def\SL{\operatorname{SL}}
\def\Mat{\operatorname{Mat}}

\def\Ker{\operatorname{Ker}}

\def\Lie{\operatorname{Lie}}

\def\Ga{\bold{G}_a}
\def\Gm{\bold{G}_m}

\renewcommand{\ker}{\mathrm{Ker}}

\begin{document}
\title[Triviality of cohomologies of  differential algebraic groups]{Triviality of differential Galois cohomologies of linear differential algebraic groups}
\author{Andrei Minchenko}
\address{Department of Mathematics, University of Vienna, Austria}
\email{an.minchenko@gmail.com}
\urladdr{http://www.mat.univie.ac.at/~minchenko/}
\author{Alexey Ovchinnikov}
\address{CUNY Queens College, Department of Mathematics,
65-30 Kissena Blvd, Queens, NY 11367 and
CUNY Graduate Center, Ph.D. programs in Mathematics and Computer Science, 365 Fifth Avenue,
New York, NY 10016}
\email{aovchinnikov@qc.cuny.edu}
\urladdr{http://qc.edu/~aovchinnikov/}
\begin{abstract}We show that the triviality of the differential Galois cohomologies over a partial differential field $K$ of a linear differential algebraic group is equivalent to $K$ being algebraically, Picard--Vessiot, and linearly differentially closed. This former is also known to be equivalent to the uniqueness up to an isomorphism of a Picard--Vessiot extension of a linear differential equation with parameters.
\end{abstract}
\maketitle

\section{Introduction} 
Galois theory of linear differential equations with parameters \cite{cassisinger} (also known as the parameterized Picard--Vessiot theory) provides theoretical and algorithmic tools to study differential algebraic dependencies of solutions of a linear ODE with one or several parameters. A parameterized Picard--Vessiot extension is a differential field generated by a complete set of solutions of the ODE and also satisfying additional technical conditions. How to decide whether such an extension is unique is an open problem in this theory. We study this question in the present paper as follows.

Let $F$ be a differential field of characteristic zero with commuting derivations  $\{\partial_x,\partial_1,\ldots,\partial_m\}$. One can show using~\cite{GGO} that the uniqueness up to an isomorphism of a Picard--Vessiot extension of any parameterized linear differential equation  with coefficients in $F$ is equivalent to  the triviality of differential Galois (also known as constrained) cohomologies \cite{Kolchin:differentialalgebraicgroups} over $K$, the $\partial_x$-constants of $F$, of all linear differential algebraic groups \cite{Cassidy:differentialalgebraicgroups}. We show in our main result, Theorem~\ref{thm:main}, that the latter triviality holds if and only if $K$ is algebraically, Picard--Vessiot, and linearly differentially closed (the terminology is explained in Section~\ref{sec:defs}).

Such a question for $m=1$ was settled in  \cite{PillayH1}, in which case the ``linearly differentially closed'' condition does not play a role. This was extended to $m>1$ in~\cite{CP2017} in terms of generalized strongly normal extensions. Our characterization is different and our arguments can be viewed as more transparent.
\section{Definitions and notation}\label{sec:defs}
\begin{definition} A {\em differential ring} is a ring $R$ with a finite set $\Delta=\{\delta_1,\ldots,\delta_m\}$ of commuting derivations on $R$. A differential ideal of $(R,\Delta)$
is an ideal of $R$ stable under any derivation in $\Delta$. \end{definition}

 For any derivation $\delta :R\to R$, we denote $
 R^\delta= \{r \in R\:|\: \delta(r) = 0\}$,
 which is a $\delta$-subring  of $R$ and is called the {\em ring of $\delta$-constants} of $R$. If $R$ is a field and $(R,\Delta)$ is a differential ring, then $(R,\Delta)$ is called a {\em differential field}.  The notion of {\em differential algebra} over $(R,\Delta)$ is defined analogously. 
 
 \begin{definition}
 An ideal $I$ of $R$ is called a differential ideal of $(R,\Delta)$ if, for all $\delta\in\Delta$ and $r \in I$, $\delta(r)\in I$.
\end{definition}

The ring of $\Delta$-differential polynomials $K\{x_1,\ldots,x_n\}$ in the differential indeterminates  $x_1,\ldots,x_n$ and with coefficients in a $\Delta$-field $(K,\Delta)$ is the ring of polynomials in the indeterminates formally denoted $$\left\{\delta_1^{i_1}\cdot\ldots\cdot\delta_m^{i_m} x_i\:\big|\: i_1,\ldots,i_m\geqslant 0,\, 1\leqslant i\leqslant n\right\}$$ with coefficients in $K$. We endow this ring with a structure of $K$-$\Delta$-algebra by setting 
$$\delta_k \left(\delta_1^{i_1}\cdot\ldots\cdot \delta_m^{i_m} x_i \right)= \delta_1^{i_1} \cdot\ldots\cdot \delta_k^{i_k+1} \cdot\ldots\cdot \delta_m^{i_m} x_i.$$ 

\begin{definition}[{see \cite[Corollary~1.2(ii)]{Marker2000}}]
A differential field $(K,\Delta)$ is said to be {\em differentially closed} if, for every $n\geqslant 1$ and every finite set of $\Delta$-polynomials $F \subset K\{x_1,\ldots,x_n \}$, if the system of differential equations $F=0$ has a solution  with entries in some $\Delta$-field extension $L$, then it has a solution with entries in $K$.\end{definition}

Let $\U$ be a differentially closed $\Delta$-field  of characteristic $0$ and  $K\subset\U$ be its differential subfield.

\begin{definition}\label{def:Kc} A {\it Kolchin-closed}  set $W \subset \U^n$ over $K$   is the set of common zeroes
of a system of differential polynomials with coefficients in $K$, that is, there exists  $S \subset K\{y_1,\dots,y_n\}$ such that
$$
W = \left\{ a \in \U^n\:|\: f(a) = 0 \mbox{ for all } f \in S \right\}.$$
 More generally, for a differential algebra $R$ over $K$, $$
W(R) = \left\{ a \in R^n\:|\:   f(a) = 0 \mbox{ for all } f \in S \right\}.$$
\end{definition}

\begin{definition}
If $W \subset \U^n$ is a  Kolchin-closed set defined  over $K$,  the differential ideal  $$\I(W) = \{ f\in K\{y_1,  \ldots , y_n\} \ | \ f(w) = 0 \mbox{ for all } \ w\in W(\U)\}$$
is called the {\em defining ideal} of $W$ over $K$.
Conversely, for a  subset  $S$ of $K\{y_1,\dots,y_n\}$, the following subset  is Kolchin-closed in  $\U^n$ and defined over $K$:
$$
\bold{V}(S)=  \left\{ a \in \U^n\:|\: f(a)= 0 \mbox{ for all } f \in S \right\}.$$ 
\end{definition}

\begin{definition}
Let $W \subset \U^n$ be a Kolchin-closed set defined over $K$. The {\em coordinate ring} $K\{W\}$ of 
$W$ over $K$ is the differential algebra over $K$
$$
K\{W\} = K\{y_1,\ldots,y_n\}\big/\I(W).
$$
If $K\{W\}$ is an integral domain, then $W$ is said to be {\it irreducible}. This  is equivalent to $\I(W)$ being a prime  differential ideal.
\end{definition}

\begin{definition}
Let $W \subset \U^n$ be a Kolchin-closed set defined over $K$. 
Let  $\I(W) = \mathfrak{p}_1\cap\ldots\cap \mathfrak{p}_q$ be  a minimal differential prime decomposition of  $\I(W)$, that is, the $\mathfrak{p}_i \subset K\{y_1,\dots,y_n\}$ are prime differential ideals containing $ \I(W)$ and minimal with this property. This decomposition is unique  up to permutation (see \cite[Section~VII.29]{Kapldiffalg}).  The irreducible Kolchin-closed sets 
$W_i=\bold{V}(\mathfrak{p}_i)$ are defined over $K$ and  called the {\it irreducible components} of $W$. We  have $W = W_1\cup\ldots\cup W_q$. 
\end{definition}

\begin{definition}
Let $W_1 \subset \U^{n_1}$ and $W_2 \subset \U^{n_2}$ be two Kolchin-closed sets defined over $K$. 
A  {\em differential polynomial map} (morphism) defined over $K$ is a map  $$\varphi : W_1\to W_2,\quad a \mapsto \left(f_1(a),\dots,f_{n_2}(a)\right),\ \  a \in W_1\,,$$ where $f_i \in K\{x_1,\dots,x_{n_1}\}$ for all $i=1,\dots,n_2$.  

If $W_1 \subset W_2$, the inclusion map of $W_1$ in $W_2$ is a differential polynomial map. In this case, we say that $W_1$ is 
a  Kolchin-closed subset of $W_2$.
\end{definition}

Let $W$ be an irreducible Kolchin-closed set and $P \subset K\{x_1,\ldots,x_n\}$ be its defining differential ideal, which is prime. It is shown in \cite[Section~II.12]{Kol} that there exists a non-negative integer $H$ such that, for all $h\geqslant H$,
$$
\dim \big(P\cap K\big[\delta_1^{i_1}\cdot\ldots\cdot\delta_m^{i_m}x_i\:|\: 1\leqslant i\leqslant n,\ i_j\geqslant 0,\,1\leqslant j\leqslant m,\ i_1+\ldots +i_m \leqslant h\big]\big)
$$
coincides with a polynomial in $h$. The degree of this polynomial is denoted by $\tau(W)$ and called the {\em differential type} of $W$ (if $W$ is a single point and so the above polynomial is 0, we set $\tau(W)=-1$). 

\begin{example}
Let $\GL_n \subset \U^n$ be the group of $n \times n$ invertible matrices   with entries in $\U$. One can see
$\GL_n$ as a Kolchin-closed subset of $\U^{n^2} \times \U$ defined over $K$, defined by the equation $x\cdot\det(X)-1$ in $K\big\{\U^{n^2} \times \U\big\}=K\{X,x\}$, where $X$ is an $n \times n$-matrix of  differential indeterminates over $K$ and $x$ a  differential indeterminate over $K$. One can thus identify the differential coordinate ring of $\GL_n$ over $K$ with  $F\{X,1/\det(X)\}$, where $X=(x_{i,j})_{1 \leqslant i,j \leqslant n} $ is
a matrix of differential indeterminates over $K$. We also denote  the special linear group that consists of the matrices  of determinant $1$ by $\SL_n \subset \GL_n$.
\end{example}

\begin{definition}[{\cite[Chapter~II, Section~1, page~905]{Cassidy:differentialalgebraicgroups}\label{def:LDAG}}] A {\em linear differential algebraic group} (LDAG) $G \subset \U^{n^2}$ defined  over $K$
is a subgroup of   $\GL_n$ that is a Kolchin-closed set defined over $K$. If $G \subset H \subset \GL_n$ are Kolchin-closed subgroups of 
$\GL_n$, we say that $G$ is a Kolchin-closed subgroup of $H$.
\end{definition}

\begin{definition}
Let $G$ be an LDAG  defined over $F$. The irreducible component of $G$ containing the identity element $e$  is called the {\it identity component} of $G$ and denoted by $G^\circ$. The LDAG $G^\circ$ is a $\delta$-subgroup of $G$ defined over $F$. The LDAG
$G$ is said to be {\it connected} if $G = G^\circ$, which is equivalent to $G$ being an irreducible Kolchin-closed set \cite[page~906]{Cassidy:differentialalgebraicgroups}.
\end{definition}

\begin{definition}[{\cite[Definition~2.6]{CassSingerJordan}}]
Let $G$ be an LDAG over $K$. The {\em strong identity component} $G_0$ of $G$ is defined to be the smallest  differential algebraic subgroup $H$ of $G$ defined over $\U$ such that $\tau(G/H) < \tau(G)$. 

By \cite[Remark~2.7(2)]{CassSingerJordan}, $G_0$ is a normal subgroup of $G$ defined over $K$.
\end{definition}

\begin{definition}[{\cite[Definition~2.10]{CassSingerJordan}}]  An infinite  LDAG $G$ defined over $K$ is {\em almost simple} if, for any normal proper  differential algebraic subgroup $H$ of $G$ defined over $K$, we have $\tau(H) < \tau(G)$.
\end{definition}

\begin{definition} For a system $\cS$ of $\Delta$-differential equations over $K$, 
$K$ is said to be \emph{$\cS$-closed}, or \emph{closed w.r.t. $\cS$}, if the consistency of $\cS$ (i.e., the existence of a solution in $\U$) implies the existence of a solution in $K$.
\end{definition}

\begin{definition}$K$ is said to be \emph{PV closed} if, for all $r$, $1\leqslant r\leqslant m$, for all  sets $\{D_1,\ldots,D_r\}\subset K\Delta$ of commuting derivations, for all $n\geqslant 1$, and for all $A_1,\ldots,A_r\in \Mat_{n\times n}(K)$, $K$ is closed w.r.t. 
\begin{equation}\label{eq:1}\{D_i(Z)=A_i\cdot Z,\ z\cdot\det Z=1\}_{i=1}^r,
\end{equation} where $Z$ and $z$ are unknown matrices of sizes $n\times n$ and $1\times 1$, respectively (see \cite[page~51]{GO} for a coordinate-free definition).
\end{definition}

\begin{definition}$K$ is said to be \emph{$\Delta$-linearly closed}  if it is closed w.r.t. any system of linear (not necessarily homogeneous) $\Delta$-differential equations in one unknown over $K$.
\end{definition}

\begin{definition} A {\em principal homogeneous space} (PHS) over an LDAG $G$ over $K$ is a Kolchin-closed  $X$ defined over $K$ together with a differential algebraic isomorphism
$X\times G \to X\times X$ over $K$.
\end{definition}
The set of equivalence classes of PHS of $G$ over $K$ is denoted by $H^1_{\Delta}(K,G)$.
We write $H^1_{\Delta}(K,G)=\{1\}$ if all principal homogeneous spaces of $G$ are isomorphic over $K$. For example, $H^1_{\Delta}(\U,G)=\{1\}$.

\section{Main result}
\begin{theorem}\label{thm:main}
The following are equivalent:
\begin{enumerate}
\item[(1)] $K$ is algebraically closed, PV closed, and $\Delta$-linearly closed;
\item[(2)] for any linear differential algebraic group $G$, $H^1_{\Delta}(K,G)=\{1\}$.
\end{enumerate}
\end{theorem}

\begin{proof}
Let us show the implication $\Longleftarrow$. If $K$ is not algebraically closed, then there exists a non-trivial Galois extension $E/K$ given by an irreducible polynomial $f$. The set $X$ of its roots is a $K$-torsor for $G=Gal(E/K)$. It is non-trivial since there are no $K$-points of $X$, that is, homomorphisms $E\to K$ over $K$. Hence,  $H^1(K,G)\cong H^1_{\Delta}(K,G)\neq\{1\}$ (with the isomorphism following from~\cite[p.~177, Theorem~4]{Kolchin:differentialalgebraicgroups}). 

Suppose that  $K$ is not PV closed. Hence, there exists a set $D=\{D_1,\ldots,D_r\}\subset K\Delta$ of commuting derivations and a system \eqref{eq:1} with no solutions in $\GL_n(K)$. 
We claim that $H^1_{\Delta}(K,\GL_n^D)\neq\{1\}$. Indeed, let 
$$
J:=\{(B_1,\ldots,B_r)\in\mathfrak{gl}_n(\U)^r\ :\ D_iB_j-D_jB_i=[B_j,B_i]\}
$$
and
$$
\ell: \GL_n\to J,\qquad x\mapsto(x^{-1}D_1(x),\ldots,x^{-1}D_r(x)).
$$  We have $\ker(\ell)=\GL_n^D$. 
Moreover, by \cite[Proposition 14, p.~26]{Kolchin:differentialalgebraicgroups}, $\ell$ is surjective. Hence, the sequence
$$
\begin{CD}
\{1\}@>>> \GL_n^D@>>> \GL_n@>\ell>> J@>>> \{0\}
\end{CD}
$$
is exact. By assumption, $\ell$ is not surjective on $K$-points. Let $(A_1,\ldots,A_n)\notin\ell(\GL_n(K))$. By \cite[p.~192, Proposition~8]{Kolchin:differentialalgebraicgroups}, $\ell^{-1}(A_1,\ldots,A_n)$ is a non-trivial torsor for~$\GL_n^D(K)$.

If  $K$ is not $\Delta$-linearly closed, then there exist a positive integer $r$, 
a $\Delta$-subgroup $B\subset\Ga^r$ defined over $K$, and a surjective $\Delta$-linear map $\Lambda: \Ga\to B$ over $K$ that is not surjective on $K$-points. By \cite[p.~192, Proposition~8]{Kolchin:differentialalgebraicgroups}, $H^1_{\Delta}(K,\Ker\Lambda)\neq\{1\}$.

Let us prove the implication $\Longrightarrow$. By \cite[p. 170, Theorem 2]{Kolchin:differentialalgebraicgroups}, given a short exact sequence
$$
\begin{CD}
\{1\}@>>> G'@>>> G@>>> G''@>>>\{1\}
\end{CD}
$$
of LDAGs over $K$ in which $G'\subset G$ is normal \cite[p.~169]{Kolchin:differentialalgebraicgroups}, 
\begin{equation}\label{eq:exactseq}
H^1_{\Delta}(K,G')=H^1_{\Delta}(K,G'')=\{1\}\implies H^1_{\Delta}(K,G)=\{1\}.
\end{equation} 
This is called the \emph{inductive principle}~\cite{PillayH1}. As in~\cite{PillayH1}, the problem reduces to the following three cases:
\begin{enumerate}
\item\label{case0} $G$ is finite;
\item\label{case1} $G\subset\Ga$;
\item\label{case2} $G=\Gm(\cC)$ --- the group of constants of $\Gm$;
\item\label{case3} $G=H^P$, where $H$ is a 
linear algebraic group (LAG) over $\Q$, $P\subset K\Delta$ is a Lie subspace, and $H^P$ is the functor of taking constant points with respect to $P$: $H^P(L):=H(L^P)$ for a $\Delta$-ring extension $L$ of $K$.  { Note that case~\eqref{case2} is included into this case, but we have separated case~\eqref{case2} for the clarity of the presentation.}
\end{enumerate}
Let us {explain} the {the reduction and then show that $H^1_{\Delta}(K,\wt{G})=1$ for any $\wt{G}$ of types~\eqref{case0}--\eqref{case3}}. The exact sequence
$$
\begin{CD}
\{1\}@>>> G^\circ@>>> G@>>> G/G^\circ@>>>\{1\}
\end{CD}
$$
reduces the problem to case~\eqref{case0} and { to showing that, for a connected~$G$, $H^1_{\Delta}(K,G)=1$}. 
To show the latter equality, let us use induction on the differential type $\tau(G)$, the case $\tau(G)=-1$ (that is, $G=\{1\}$, because $G$ is assumed to be connected) being trivial. Let $G_0\subset G$ be the strong identity component. 
 Suppose $\tau(G)\geqslant 0$. One has the following exact sequence:
$$
\begin{CD}
\{1\}@>>> G_0@>>> G@>>> G/G_0@>>>\{1\},
\end{CD}
$$ 
and $\tau(G/G_0)<\tau(G)$. By the induction, this reduces the problem to the case $G=G_0$. Moreover, by \cite[Theorem 2.27 and Remark 2.28(2)]{CassSingerJordan}, it suffices to assume that $G$ is almost simple.  
If $G$ is almost simple non-commutative, it is simple by~\cite[Theorem 3]{MinchenkoCentral}. By \cite[Theorems 9 and 17]{Cassimpl}, it corresponds to case~\eqref{case3} because $K$ is PV closed. If $G$ is commutative, $\ol{G}$ is also commutative, hence
there are $n_1,n_2\geqslant 0$ such that $\ol{G}$ is isomorphic over $K$ (recall that $K$ is algebraically closed) to a direct product of $n_1$ copies of $\Ga$ and $n_2$ copies of $\Gm$ (we will not use the almost simplicity in the commutative case). It follows by induction on $n_1+n_2$  that $H^1_{\Delta}(K,G)=1$ if $H^1_{\Delta}(K,\wt{G})=1$ for any connected Kolchin closed subgroup $\wt{G}$ of $\Ga$ or $\Gm$. Indeed, if $n_1\geqslant 1$, one has a natural projection $G\subset \ol{G}\to\Ga$ whose kernel is contained in the direct product of $n_1-1$ copies of $\Ga$ and $n_2$ copies of $\Gm$. Similarly, if $n_2\geqslant 1$, one considers a projection $G\subset\ol{G}\to\Gm$.

The case $G\subset\Gm$ reduces to $G=\Gm^{K\Delta}=\Gm(\cC)$ and case~\eqref{case1} by considering the logarithmic derivatives (defined on $\Gm$) $$\ell_i:G\to\Ga,\quad x\mapsto x^{-1}\partial_ix,$$  for $i=1,\ldots, m$ subsequently, as all infinite differential algebraic subgroups of $\Gm$ contain $\Gm^{K\Delta}$ \cite[Proposition~31]{Cassidy:differentialalgebraicgroups}. 

In case~\eqref{case0}, $H^1_{\Delta}(K,G)=\{1\}$ by \cite[p.~177, Theorem~4]{Kolchin:differentialalgebraicgroups}, because $K$ is algebraically closed. In case~\eqref{case1}, we have the following exact sequence:
$$
\begin{CD}
\{0\}@>>> G@>\iota>>\Ga @>\pi>> \Ga @>>>\{0\},
\end{CD}
$$
and we have $H^1_{\Delta}(K,G) = \{1\}$ by \cite[p.~193, Corollary~1]{Kolchin:differentialalgebraicgroups} since $K$ is linearly $\Delta$-closed.
Case~\eqref{case2} is included into case~\eqref{case3}, as noted before. 
It remains to consider case~\eqref{case3}.  
Choose a basis $\{D_1,\ldots,D_r\}$ of commuting derivations of $P$ and let 
\begin{equation}\label{eq:J}
J:=\{(B_1,\ldots,B_r)\in(\Lie H)^r\ :\ D_iB_j-D_jB_i=[B_j,B_i]\}
\end{equation}
and, since $H$ is defined over $\Q$, by \cite[p.~924, Corollary]{Cassidy:differentialalgebraicgroups}, we have:
$$
\ell: H\to J,\qquad x\mapsto(x^{-1}D_1(x),\ldots,x^{-1}D_r(x)).
$$  We have $\ker(\ell)=H^P=G$.  Let $B_1,\ldots,B_r \in J$.
Since \cite[Lemma~1]{Kovacic1969} can be rewritten in a straightforward way for several commuting derivations, the surjectivity of  $\ell$ is implied by 
\begin{equation}\label{eq:det}p \notin [D_1(x)-xB_1,\ldots,D_r(x)-xB_r]
\end{equation} 
for any order $0$ non-zero differential polynomial $p$ in $x$ (which includes $p=\det x$), as in the proof \cite[Proposition~6]{Kovacic1969}.
Since \eqref{eq:det} is shown in the proof of \cite[Proposition~14, p.~26]{Kolchin:differentialalgebraicgroups} given the conditions in~\eqref{eq:J}, we conclude that $\ell$ is surjective. Hence, the sequence
$$
\begin{CD}
\{1\}@>>> G@>>> H@>\ell>> J@>>> \{0\}
\end{CD}
$$
is exact. Since $K$ is algebraically closed, $H^1_\Delta(K,H)=\{1\}$. By \cite[p.~192, Proposition 8]{Kolchin:differentialalgebraicgroups}, this implies that all torsors for $G$ are isomorphic to $\ell^{-1}(a)$, $a\in J(K)$. Moreover, all of them are isomorphic (that is, $H^1_{\Delta}(K,G)=\{1\}$) if $\ell$ is surjective on $K$-points. 

Due to the PV closedness, $H^1_{\Delta}(K,\GL_n^P)=\{1\}$. 
Since $H$ is defined over $\mathbb{Q}$, taking $P$-constants, which can be viewed as applying the functor $\Hom(\,\cdot\,,\U^P)$, is exact, because $\U^P$ is algebraically closed and so the polynomial map $\pi$ is surjective:
$$
\begin{CD}
\{1\}@>>>G=H^P@>\iota>>\GL_n^P@>\pi>>(\GL_n/H)^P@>>> \{1\},
\end{CD}
$$
Since $K$ is algebraically closed, the map $\pi$ is surjective on $K$-points. Therefore, by the corresponding exact sequence of cohomologies \cite[p. 170, Theorem 2]{Kolchin:differentialalgebraicgroups}, $H^1_\Delta(K,G)=H^1_\Delta(K,\GL_n^{P})=\{1\}$. 
\end{proof}

\section*{Acknowledgments}
The authors are grateful to Anand Pillay and Michael F. Singer for the discussions and comments. This work has been partially supported by the NSF grants CCF-0952591 and DMS-1413859, and   by the Austrian Science Foundation FWF, grant P28079.
\bibliographystyle{abbrvnat}
\bibliography{bibdata}
\end{document}